\documentclass[reqno,11pt]{amsart}

\usepackage[a4paper,left=23mm,right=23mm,top=30mm,bottom=30mm,marginpar=25mm]{geometry}
\usepackage{amsmath}
\usepackage{amssymb}
\usepackage{amsthm}
\usepackage{amscd}
\usepackage{mathtools}
\usepackage{dsfont}
\usepackage[ansinew]{inputenc}
\usepackage{color}
\usepackage[final]{graphicx}
\usepackage{esint} 
\usepackage{bbm}
\usepackage{subfigure}
\usepackage{tikz}
\usepackage{hyperref}
\usepackage{cite}

\renewcommand{\epsilon}{\varepsilon}

\numberwithin{equation}{section}

\newtheoremstyle{thmlemcorr}{10pt}{10pt}{\itshape}{}{\bfseries}{.}{10pt}{{\thmname{#1}\thmnumber{ #2}\thmnote{ (#3)}}}
\newtheoremstyle{thmlemcorr*}{10pt}{10pt}{\itshape}{}{\bfseries}{.}\newline{{\thmname{#1}\thmnumber{ #2}\thmnote{ (#3)}}}
\newtheoremstyle{defi}{10pt}{10pt}{\itshape}{}{\bfseries}{.}{10pt}{{\thmname{#1}\thmnumber{ #2}\thmnote{ (#3)}}}
\newtheoremstyle{remexample}{10pt}{10pt}{}{}{\bfseries}{.}{10pt}{{\thmname{#1}\thmnumber{ #2}\thmnote{ (#3)}}}
\newtheoremstyle{ass}{10pt}{10pt}{}{}{\bfseries}{.}{10pt}{{\thmname{#1}\thmnumber{ A#2}\thmnote{ (#3)}}}

\theoremstyle{thmlemcorr}
\newtheorem{theorem}{Theorem}
\numberwithin{theorem}{section}
\newtheorem{lemma}[theorem]{Lemma}
\newtheorem{corollary}[theorem]{Corollary}
\newtheorem{proposition}[theorem]{Proposition}

\theoremstyle{thmlemcorr*}
\newtheorem{theorem*}{Theorem}
\newtheorem{lemma*}[theorem]{Lemma}
\newtheorem{corollary*}[theorem]{Corollary}
\newtheorem{proposition*}[theorem]{Proposition}
\newtheorem{problem*}[theorem]{Problem}
\newtheorem{conjecture*}[theorem]{Conjecture}

\theoremstyle{defi}
\newtheorem{definition}[theorem]{Definition}

\theoremstyle{remexample}
\newtheorem{remarkx}[theorem]{Remark}
\newenvironment{remark}
  {\pushQED{\qed}\remarkx}
  {\popQED\endremarkx}
\newtheorem{examplex}[theorem]{Example}
\newenvironment{example}
  {\pushQED{\qed}\examplex}
  {\popQED\endexamplex}

\theoremstyle{ass}

\newcommand{\Fcal}{\mathcal{F}}

\newcommand{\Hcal}{\mathcal{H}}

\newcommand{\Mcal}{\mathcal{M}}

\newcommand{\Sbb}{\mathbb{S}}

\DeclareMathOperator{\supp}{supp}

\newcommand{\N}{\mathbb{N}}
\newcommand{\R}{\mathbb{R}}

\def\XXint#1#2#3{{\setbox0=\hbox{$#1{#2#3}{\int}$}
\vcenter{\hbox{$#2#3$}}\kern-.5\wd0}}

\DeclarePairedDelimiter\abs{\lvert}{\rvert}
\DeclarePairedDelimiter{\norm}{\lVert}{\rVert}
\DeclarePairedDelimiter{\inner}{\langle}{\rangle}

\newcommand{\Rn}{\R^{n}}

\renewcommand{\phi}{\varphi}

\newcommand{\dx}{\, dx}
\newcommand{\dy}{\, dy}

\newcommand{\F}{\mathcal{F}}

\newcommand{\Cc}{C_{c}^{\infty}}
\newcommand{\Ccr}{C_{c}^{\infty}(\R^n)}

\newcommand{\starto}{\stackrel{*}{\rightharpoonup}}

\newcommand{\na}{\nabla^{\alpha}}
\newcommand{\da}[1]{(-\Delta)^{#1}}
\newcommand{\nanl}{\nabla^{\alpha}_{\normalfont{\text{NL}}}}

\newcommand{\diva}{\mathrm{div}^{\alpha}}

\def\XXint#1#2#3{{\setbox0=\hbox{$#1{#2#3}{\int}$}
     \vcenter{\hbox{$#2#3$}}\kern-.5\wd0}}

\newcommand{\dac}{(-\Delta)^{\frac{1-\alpha}{2}}}

\newcommand{\Rmn}{\mathbb{R}^{m \times n}}

\newcommand{\Spgm}{S^{\alpha,p}_{g}(\Omega;\R^m)}

\newcommand{\Frel}{\mathcal{F}^{\mathrm{rel}}_{\alpha}}

\newcommand{\Lipb}{\mathrm{Lip}_{b}}
\newcommand{\Lip}{\mathrm{Lip}}
\newcommand{\Da}{D^{\alpha}}
\newcommand{\bva}{BV^{\alpha}}
\newcommand{\bvag}{BV^{\alpha}_g}
\newcommand{\bvagom}{BV^{\alpha}_g(\Omega;\R^m)}
\newcommand{\Sone}{S^{\alpha,1}}

\newcommand{\Sonegom}{S^{\alpha,1}_g(\Omega;\R^m)}
\newcommand{\Ureplace}{U}

\makeatletter
\g@addto@macro\bfseries{\boldmath}
\makeatother

\title[Extending Fractional Linear Growth Functionals]{Extending Linear Growth Functionals to Functions of Bounded Fractional Variation}

\author{Hidde Sch\"{o}nberger}
\address{Mathematisch-Geographische Fakult\"at, Katholische Universit\"at Eichst\"att-Ingolstadt, Ostenstra{\ss}e 28, 85072 Eichst\"att}
\email{hidde.schoenberger@ku.de}

\begin{document}


\begin{abstract}
In this paper we consider the minimization of a novel class of fractional linear growth functionals involving the Riesz fractional gradient. These functionals lack the coercivity properties in the fractional Sobolev spaces needed to apply the direct method. We therefore utilize the recently introduced spaces of bounded fractional variation and study the extension of the linear growth functional to these spaces through relaxation with respect to the weak* convergence. Our main result establishes an explicit representation for this relaxation, which includes an integral term accounting for the singular part of the fractional variation and features the quasiconvex envelope of the integrand. The role of quasiconvexity in this fractional framework is explained by a technique to switch between the fractional and classical settings. We complement the relaxation result with an existence theory for minimizers of the extended functional.

\vspace{8pt}

\vspace{8pt}

\vspace{8pt}

\noindent\textsc{Date:} \today.
\end{abstract}
\maketitle
\thispagestyle{empty}
%

\section{Introduction}
Motivated from both the practical and theoretical point of view, the study of nonlocal aspects in the calculus of variations has received widespread attention in the literature recently. From applications in peridynamics~\cite{Sil00, MeD15}, imaging processing \cite{OsherGilboa, AubertKornprobst, Antiletal22} and machine learning~\cite{AntilKhatri, HollerKunish}, to the abstract study of lower semicontinuity~\cite{Pedregalperidynamics,Bellidoperidynamics,KRZ22, KrS22} and localization~\cite{Bellido2, BMCP15, AAB20} of various nonlocal functionals. Especially the introduction of the so-called Riesz fractional gradient by Shieh \& Spector \cite{Shieh1,Shieh2}, which for $\phi \in C_c^{\infty}(\R^n)$ and $\alpha \in (0,1)$ is defined as
\[
\na \phi(x) =\mu_{n,\alpha}\int_{\R^n}\frac{\phi(y)-\phi(x)}{\abs{y-x}^{n+\alpha}}\frac{y-x}{\abs{y-x}}\,dy \quad \text{for $x \in \R^n$},
\]
has seen a dramatic rise in interest and has opened up the possibility to study new types of fractional problems. We refer to just a few of the recent works \cite{MengeshaUnified,Unified,BCMC22, KrS22, Bellido}. The Riesz fractional gradient provides an alternative to the more well-known fractional Laplacian and shares many similarities with the classical gradient. In fact, it is the unique translationally and rotationally invariant $\alpha$-homogeneous operator \cite{Silhavy}, which makes it a canonical choice for a fractional gradient.

The definition of the fractional gradient can be extended in a distributional way to define the naturally associated fractional Sobolev spaces
\[
S^{\alpha,p}(\R^n;\R^m)=\{ u \in L^p(\R^n;\R^m)\,:\, \na u \in L^p(\R^n;\Rmn)\},
\]
with $\alpha \in (0,1)$ and $p \in [1,\infty]$, see \cite{KrS22, comi1, comi2, comi3} for more details. With these new spaces came an inherent class of variational problems to study, that is, integral functionals depending on the Riesz fractional gradient. Precisely, with $\Omega \subset \R^n$ open and bounded, $p \in (1,\infty)$ and $g \in S^{\alpha,p}(\R^n;\R^m)$, one defines the functions subjected to a typical complementary-value condition
\[
S^{\alpha,p}_g(\Omega;\R^m)=\{u \in S^{\alpha,p}(\R^n;\R^m)\,:\, u=g \ \text{a.e.~in $\Omega^c$}\},
\]
and aims to minimize the functional
\begin{equation}\label{eq:pfunc}
\Spgm \ni u \mapsto \int_{\R^n} f(x,\na u(x))\,dx;
\end{equation}
here $f:\R^n\times \Rmn \to \R$ is a Carath\'{e}odory integrand with suitable $p$-growth and coercivity bounds. 

The weak lower semicontinuity and existence of minimizers of these functionals was initially shown in the scalar setting in \cite{Shieh1, Shieh2} under the condition of convexity in the second argument of $f$ and later extended to the vectorial case under polyconvexity in \cite{Bellido}. More recently, in \cite{KrS22} the weak lower semicontinuity of the functional $\Fcal_{\alpha}$ was fully characterized in terms of the notion $\alpha$-quasiconvexity, which is a condition on a function $h: \Rmn \to \R$ that requires that
\[
h(A) \leq \int_{(0,1)^n}h(A+ \na \phi(y))\,dy \quad \text{for all $A \in \Rmn$ and $\phi \in C^{\infty}_{\rm per}((0,1)^n;\R^m)$},
\]
see \cite[Definition~4.6]{KrS22}. The proof of this result relied on a method to translate fractional gradients into classical gradients and back by using the identities
\begin{equation}\label{eq:relation}
\na \phi = \nabla I_{1-\alpha} \phi \quad \text{and} \quad \nabla \phi = \na \dac \phi \quad \text{for $\phi \in C_c^{\infty}(\R^n)$},
\end{equation}
and actually revealed that the notion of $\alpha$-quasiconvexity is independent of $\alpha \in (0,1)$ and equivalent to Morrey's well-known quasiconvexity \cite{Mor}. Therefore, the weak lower semicontinuity of the functionals $\Fcal_{\alpha}$ can be characterized in the same way as the classical integral functionals in the calculus of variations. \medskip

Inspired by the rich history on classical linear growth problems, cf.~\cite{GoS64,Res68,Dal80,AmD92,FoM93,KrR10a,RiS19}, we build upon the above results and exploit the distributional character of the fractional Sobolev spaces to consider the first class of fractional linear growth functionals in the literature. This class constitutes the natural extension of \eqref{eq:pfunc} to $p=1$, namely, functionals of the form
\begin{equation}\label{eq:functional1}
\Fcal_{\alpha}(u) = \int_{\R^n} f(x,\na u(x))\,dx \quad \text{for $u \in \Sonegom$},
\end{equation}
with $f:\R^n \times \Rmn \to \R$ a linear growth Carath\'{e}odory integrand and $g \in S^{\alpha,1}(\R^n;\R^m)$.

The immediate difficulty in the minimization of the above functional is the non-reflexivity of $\Sonegom$, which prevents the direct method from being used with respect to the weak convergence in $S^{\alpha,1}(\R^n;\R^m)$. Therefore, taking up a similar approach as in the classical case, one can suitably extend the functional $\Fcal_{\alpha}$ to a larger space of bounded fractional variation, in which compactness holds with respect to the weak* convergence. 

These spaces of bounded fractional variation and their properties have already been thoroughly studied by Comi \& Stefani and coauthors in \cite{comi1, comi2, comi3} and can be understood as
\[
BV^{\alpha}(\R^n;\R^m) =\{ u \in L^1(\R^n;\R^m)\,:\, \Da u \in \Mcal(\R^n; \Rmn)\},
\]
with $\Da u$ the so-called fractional variation measure of $u$ defined in a distributional sense. We also use the notation
\[
D^{\alpha}u= \na u \,dx + D^{\alpha}_su,
\]
where $\na u \in L^1(\R^n;\Rmn)$ is the absolutely continuous part of $\Da u$ with respect to the Lebesgue measure and $\Da_s u \in \Mcal(\R^n;\Rmn)$ is the singular part.
This new class of bounded variation spaces possesses interesting similarities and differences
with the classical $BV$-spaces and has sparked a lot of further investigations. Aspects such as the description of precise representatives \cite{Comiprecise}, Leibniz rules \cite{Comileibniz}, and the failure of a local chain rule \cite{Comichain} have been considered. Very recently, the fractional total variation has been used in the context of image processing providing a nonlocal alternative to the total variation regularization \cite{Antiletal22}. 

For the sake of finding an extension of $\Fcal_{\alpha}$, we introduce the complementary-value space
\[
\bvagom = \{ u \in BV^{\alpha}(\R^n;\R^m)\,:\, u=g \ \text{a.e.~in $\Omega^c$}\};
\]
bounded sequences in $\Sonegom$ will converge up to subsequence to an element of $\bvagom$ with respect to the weak* convergence (see Section~\ref{sec:bv}). Therefore, with an eye towards minimization, the natural extension of $\Fcal_{\alpha}$ to $\bvagom$ is the relaxation defined by
\begin{equation}\label{eq:reldef}
\Frel(u)=\inf\left\{ \liminf_{j \to \infty}\Fcal_{\alpha}(u_j)\,:\, (u_j)_j \subset \Sonegom, \ u_j \starto u \ \text{in $\bvagom$}\right\}
\end{equation}
for $u \in \bvagom$. The useful features of the functional $\Frel$ are that it admits a minimizer under suitable coercivity conditions and that minimizing sequences of $\Fcal_{\alpha}$ converge up to subsequence to minimizers of $\Frel$.

To benefit from these attributes, it is key to find an explicit representation of the relaxed functional. For this, one must, in particular, account for the concentration effects that fractional gradients of sequences in $\Sonegom$ can exhibit and how they relate to the singular part of the limiting fractional variation measure. The well-known concept of the (strong) recession function (cf.~\cite{KrR10a}), which describes the way an integrand $f$ behaves at infinity, is capable of this and is defined as
\begin{equation}\label{eq:recessionfrac}
f^{\infty}(x,A) = \lim_{\substack{(x',A') \to (x,A) \\ t \to \infty}}\frac{f(x',tA')}{t} \quad \text{for $x \in \R^n$ and $A \in \Rmn$},
\end{equation}
whenever it exists. We also recall the upper recession function $f^{\#}$, which is always well-defined, by replacing the limit in \eqref{eq:recessionfrac} with a limit superior. In addition, throughout the paper we use the following growth and coercivity bounds
\begin{equation}\label{eq:growth}\tag{G}
\abs{f(x,A)} \leq M\abs{A}+a(x) \quad \text{ for all $x \in \R^n$ and $A \in \Rmn$},
\end{equation}
with $M>0$ and $a \in L^1(\R^n)\cap L^{\infty}(\R^n)$ and
\begin{equation}\label{eq:coercivity}\tag{C}
\mu \abs{A} - c \leq f(x,A) \quad \text{ for all $x \in \R^n$ and $A \in \Rmn$},
\end{equation}
with $\mu,c >0$. Note that the growth bound ensures that $f$ has linear growth and $\Fcal_{\alpha}$ is well-defined and finite. We now state the following representation result for the relaxation of $\Fcal_{\alpha}$, which is the main result of the paper.

\begin{theorem}\label{th:main}
Let $\alpha \in (0,1)$, $\Omega \subset \R^n$ be a bounded Lipschitz domain and $g \in \Sone(\R^n;\R^m)$. Assume $f:\R^n \times \Rmn \to \R$ is a Carath\'{e}odory integrand that satisfies \eqref{eq:growth} and \eqref{eq:coercivity}, and that 
\begin{equation}\label{eq:technicalcond}
\text{$f^{\infty}(x,A)$ exists and $(f^{\rm qc})^{\#}(x,A)=\limsup_{\substack{A'\to A \\ t \to \infty}} \frac{f^{\rm qc}(x,tA')}{t}$ for all $(x,A) \in \overline{\Omega} \times \Rmn,$}
\end{equation}
with $f^{\rm qc}$ the quasiconvex envelope of $f$ with respect to its second argument. Then, the relaxation of $\F_{\alpha}$ in \eqref{eq:functional1} given by \eqref{eq:reldef} can be represented as
\begin{equation}\label{eq:relaxation}
\begin{split}
\F_{\alpha}^{\rm rel}(u) =\int_{\Omega} f^{\rm qc}(x,\na u)\dx+\int_{\overline{\Omega}}(f^{\rm qc})^{\#}\left(x,\frac{d \Da_s u}{d\abs{\Da_s u}}\right)\,d\abs{\Da_s u}+\int_{\Omega^c}f(x,\na u)\dx,
\end{split}
\end{equation}
for $u \in \bvagom$. 
\end{theorem}
This theorem provides a fractional analogue to the relaxation result in the classical $BV$-setting \cite{KrR10b, ADR} and an extension of the $p$-growth fractional relaxation in \cite[Theorem~1.2]{KrS22}. The reason that the quasiconvex envelope arises in the relaxation is related to the fact that quasiconvexity is the correct characterizing notion for lower semicontinuity similarly as in the $p$-growth case from \cite{KrS22}. However, the integrand remains unchanged for $x \in \Omega^c$, since fractional gradients of weak*  convergent sequences in $\bvagom$ converge strongly in sets with a positive distance from $\Omega$ (Lemma~\ref{le:strongoutside}). Furthermore, the second integral relating to the singular part of the fractional variation is only integrated over $\overline{\Omega}$, because for $u \in \bvagom$ the measure $D^{\alpha}_s u$ is supported on $\overline{\Omega}$. This follows since the singular part of the fractional variation actually behaves locally, cf.~Remark~\ref{rem:locality}, implying that $D^{\alpha}_s u =D^{\alpha}_s g =0$ outside of $\overline{\Omega}$.
A sufficient condition for \eqref{eq:technicalcond} only in terms of $f$ is given in Remark~\ref{rem:relaxation}\,c). \medskip
 
The proof of the lower bound of the relaxation result hinges on a characterization of the weak* lower semicontinuity of functionals of the form
\begin{equation}\label{eq:extendedfunctional}
\overline{\Fcal}_{\alpha}(u) = \int_{\R^n} f(x,\na u)\,dx + \int_{\overline{\Omega}} f^{\infty} \left(x,\frac{d \Da_s u}{d\abs{\Da_s u}}\right)\,d\abs{\Da_s u} \quad \text{for $u \in \bvagom$},
\end{equation}
see Theorem~\ref{th:sufficiency}. It states that the lower semicontinuity is equivalent to $f(x,\cdot)$ being quasiconvex for a.e.~$x \in \Omega$ and is proven by using an analogue of the identities in \eqref{eq:relation} for functions of bounded variation as established in \cite{comi2}. In addition, we make substantial use of the theory of generalized Young measures developed in \cite{DiM87, AlB97, KrR10a} for linear growth problems, which allow one to capture the oscillation and concentration effects of sequences of measures. A technical issue arises from the fact that we only assume that $f^{\infty}$ exists for $x \in \overline{\Omega}$, requiring some care to account for the possible mass that comes from outside $\Omega$ and concentrates on the boundary $\partial \Omega$. 

The construction of a recovery sequence for the upper bound is carried out in two steps. We first find for $u \in \bvagom$ a sequence $(u_j)_j \subset \Sonegom$ that converges to $u$ in a strong enough sense so that the values of the functional along the sequence converge. The natural notion, which has been utilized in the classical case \cite{KrR10b}, is that of area-strict convergence (Definition~\ref{def:areastrict}). To exploit the properties of area-strict convergence, we prove that $\Sonegom$ (and even $g+C_c^{\infty}(\Omega;\R^m)$) is dense in the larger space $\bvagom$ with respect to this convergence, see Theorem~\ref{th:density}. The second step can then restrict to smooth functions to recover the quasiconvexification of the integrand and relies on adaptations of the argument in \cite[Theorem~1.2]{KrS22} and the identities in \eqref{eq:relation}. \medskip

Finally, we complement the relaxation and lower semicontinuity result with corresponding statements about the existence of minimizers under the coercivity condition \eqref{eq:coercivity}, see Corollary~\ref{cor:existence} and Remark~\ref{rem:relaxation}\,a). This actually requires an improved version of the fractional Poincar\'{e} inequality (Proposition~\ref{prop:poincare}) that only involves the fractional variation over a bounded domain. In particular, the area-integrand $f(A)=\sqrt{1+\abs{A}^2}-1$ in Example~\ref{ex:area} is an admissible candidate, providing a fractional analogue to the famous Plateau problem \cite{Giu84}.

An interesting open problem for further study is the relaxation of $\Fcal_{\alpha}$ when the integrand admits additional dependence on the values of $u$. Indeed, in the introduction of \cite{Comiprecise} it is mentioned for $u \in \bva(\R^n;\R^m)$ that $\Da u$ can be non-zero on sets of Hausdorff dimension $n-1$, just as the classical variation, while the precise representative of $u$ is only defined for $\Hcal^{n-\alpha+\epsilon}$-a.e.~$x \in \R^n$ for any $\epsilon >0$. This discrepancy between $n-1$ and $n-\alpha$ is not present in the classical case and makes it hard to deal with the singular part of the relaxation.

The structure of the text is as follows. In Section~\ref{sec:prelims} we present the notation and necessary preliminaries such as generalized Young measure theory and fractional calculus. Section~\ref{sec:bv} revolves around the spaces of bounded fractional variation and contains the proof of the density result with respect to area-strict convergence. The next section is devoted to the characterization of the weak* lower semicontinuity of extended functionals as in \eqref{eq:extendedfunctional}, and Section~\ref{sec:relaxation} rounds off the paper with the proof of Theorem~\ref{th:main}.

\section{Preliminaries}\label{sec:prelims}

\subsection{Notation}
The ball centered at $x \in \R^n$ with radius $r>0$ is denoted by $B_r(x)=\{y \in \R^n \,:\, \abs{x-y} < r\}$. The notation $E \Subset F$ for sets $E, F\subset \R^n$ means that $E$ is compactly contained in $F$, i.e.~$\overline{E} \subset F$ and $\overline{E}$ is compact. We denote by
\[
\mathbbm{1}_E(x) = \begin{cases} 1 &\text{for} \ x \in E,\\
0 &\text{otherwise},
\end{cases}\qquad x\in \R^n,
\]
the indicator function of a set $E \subset \R^n$.

By $\Lipb(\R^n)$ and $\Lip_{c}(\R^n)$, we refer to all the functions $\psi:\R^n \to \R$ that are Lipschitz continuous and bounded or Lipschitz continuous with compact support on $\R^n$, respectively; we write $\Lip(\psi)$ for the Lipschitz constant of $\psi$. Furthermore, for $X \subset \R^n$ open or closed we denote by $C_0(X)$ the Banach space obtained by taking the closure of the smooth compactly supported functions $C_c^{\infty}(X)$ with respect to the supremum norm. In particular, if $X$ is compact then $C_0(X)$ consists of all continuous functions from $X$ to $\R$. 

The space $\Mcal(X)$ consists of all finite Radon measures on $X$ and is the dual space of $C_0(X)$. As such, we say that $(\mu_j)_j \subset \Mcal(X)$ converges weak* to $\mu \in \Mcal(X)$ if $\int_X \phi \,d\mu_j \to \int_X \phi \,d\mu$ for all $\phi \in C_0(X)$. More generally, one can define for $f:X \to \R$ Borel measurable and $\mu \in \Mcal(X)$ the duality bracket $\inner{f,\mu}=\int_X f\,d\mu$. By $\Mcal^{+}(X)$ and $\Mcal^1(X)$ we denote the space of positive and probability measures, respectively. We utilize the usual notation for the Radon-Nikod\'{y}m derivative and for $\mu \in \Mcal(X)$ the Radon-Nikod\'{y}m derivative with respect to the Lebesgue measure is written as $\frac{d\mu}{dx}\in L^1(X)$, while $\mu^s \in \Mcal(X)$ represents the singular part of $\mu$ with respect to the Lebesgue measure. The measure $\abs{\mu} \in \Mcal^{+}(X)$ constitutes the total variation measure of $\mu \in \Mcal(X)$.

Finally, for $U \subset \R^n$ open we write $BV(U)$ for the space of functions of bounded variation and denote by $Du$ the total variation measure of a function $u \in BV(U)$. We use in this instance $\nabla u$ for the absolutely continuous part of $Du$ and $D_s u$ for the singular part of $Du$ with respect to the Lebesgue measure. The variant $BV_{\rm loc}(\R^n)$ consists of the function that lie in $BV(U)$ for all open and bounded $U \subset \R^n$. All of the mentioned spaces also possess vector-valued counterparts, which are denoted in the second argument like, for example, $BV(U;\R^m)$ and $\Mcal(X;\R^N)$ with $m,N \in \N$.

\subsection{Generalized Young Measures}
Generalized Young measures are a tool to study the asymptotic behavior of sequences of functions or even measures and are able to capture both the oscillation and concentration effects. Therefore, they are very well suited for studying linear growth problems in the calculus of variations. In this section we recall the basic definitions and properties that we need in the paper. We refer to \cite{KrR10a,Rindler} for more on this topic.

We begin with the definition of the (strong) recession function, which encodes the values of an integrand at infinity. For $\Ureplace \subset \R^n$ open and bounded and $f:\Ureplace \times \R^N \to \R$ it is defined as
\begin{equation*}
f^{\infty}(x,A) = \lim_{\substack{(x',A') \to (x,A) \\ t \to \infty}}\frac{f(x',tA')}{t} \quad \text{for $x \in \overline{\Ureplace}$ and $A \in \R^N$},
\end{equation*}
provided the limit exists. If the limit exists, then $f^{\infty}:\overline{\Ureplace}\times\R^N \to \R$ is automatically jointly continuous and positively homogeneous in the second argument. We now present the definition of a generalized Young measure, see~\cite{KrR10a} or \cite[Definition~2.3]{ADR}.
\begin{definition}
Let $\Ureplace \subset \R^n$ be open and bounded, then a triple $\nu=(\nu_x, \lambda_{\nu},\nu^{\infty}_x)$ is called a (generalized) Young measure on $\Ureplace$ with values in $\R^N$, we write $\nu \in Y(\Ureplace;\R^N)$, if:
\begin{itemize}
\item[$(i)$] $(\nu_x)_{x \in \Ureplace}\subset \Mcal^1(\R^N)$ is a parametrized family of probability measures on $\R^N$; \smallskip

\item[$(ii)$] $\lambda_{\nu} \in \Mcal_{+}(\overline{\Ureplace})$ is a positive measure on $\overline{\Ureplace}$; \smallskip

\item[$(iii)$] $(\nu_x^{\infty})_{x \in \overline{\Ureplace}} \subset \Mcal^1(\Sbb^{N-1})$ is a parametrized family of probability measures on $\Sbb^{N-1}$.
\end{itemize}
Additionally, it is required that $x\mapsto \inner{\abs{\cdot},\nu_x} \in L^1(\Ureplace)$ and the maps $x \mapsto \inner{f(x,\cdot),\nu_x}$ and $x \mapsto \inner{f^{\infty}(x,\cdot),\nu^{\infty}_x}$ are respectively Lebesgue measurable and $\lambda_{\nu}$-measurable for all Carath\'{e}odory integrands $f:\Ureplace\times \R^N \to \R$ for which $f^{\infty}$ exists.
\end{definition}
Intuitively, the Young measure is designed so that $(\nu_x)_{x \in \Ureplace}$ encodes the oscillations, while $\lambda^{\nu}$ determines the location and size of the concentrations, and $(\nu^{\infty}_x)_{x \in \overline{\Ureplace}}$ the direction of the concentrations. The main result about generalized Young measures is that bounded sequences of measures generate Young measures up to subsequence. Precisely, the following statement is a combination of \cite[Theorem~7 and Proposition~2]{KrR10a}.
\begin{theorem}\label{th:fundyoung}
Let $\Ureplace \subset \R^n$ be open and bounded and $(\mu_j)_j \subset \Mcal(\overline{\Ureplace};\R^N)$ a sequence such that $\sup_j \abs{\mu_j}(\overline{\Ureplace}) <\infty$. Then, there exists a subsequence (not relabeled) and a Young measure $\nu \in Y(\Ureplace;\R^N)$ with
\begin{align*}
\lim_{j \to \infty}\int_{\Ureplace} f\left(x,\frac{d\mu_j }{dx}\right)\,dx&+\int_{\overline{\Ureplace}} f^{\infty}\left(x,\frac{d\mu_j^s}{d\abs{\mu_j^s}}\right)\,d\abs{\mu_j^s} \\
&= \int_{\Ureplace} \inner{f(x,\cdot),\nu_x}\,dx + \int_{\overline{\Ureplace}} \inner{f^{\infty}(x,\cdot),\nu^{\infty}_x}\,d\lambda_{\nu},
\end{align*}
for all Carath\'{e}odory integrands $f:\Ureplace \times \R^N \to \R$ for which $f^{\infty}$ exists.
\end{theorem}
In the setting of the above theorem, we say that $(\mu_j)_j$ generates the Young measure $\nu$. We can also associate to a $\mu \in \Mcal(\overline{\Ureplace};\R^N)$ the elementary Young measure $\delta[\mu] \in Y(\Ureplace;\R^N)$ with
\begin{equation}\label{eq:elementary}
(\delta[\mu])_x = \delta_{\frac{d\mu }{dx}(x)}, \quad \lambda_{\delta[\mu]} = \abs{\mu^s} \quad \text{and} \quad (\delta[\mu])^{\infty}_x = \delta_{\frac{d\mu^s}{d\abs{\mu^s}}(x)},
\end{equation}
with $\delta_A$ the dirac measure at a point $A \in \R^N$. One can then interpret the convergence in Theorem~\ref{th:fundyoung} as weak* convergence of $\delta[\mu_j]$ to $\nu$ in $Y(\Ureplace;\R^N)$, where the duality arises from testing Young measures with integrands $f$. When we have an integrand without a well-defined strong recession function, we can still define the upper recession function
\begin{equation}\label{eq:upperrecession}
f^{\#}(x,A) = \limsup_{\substack{(x',A') \to (x,A) \\ t \to \infty}}\frac{f(x',tA')}{t} \quad \text{for $x \in \overline{\Ureplace}$ and $A \in \R^N$}.
\end{equation}
Then, if $(\mu_j)_j \in \Mcal(\overline{\Ureplace};\R^N)$ generates the Young measure $\nu \in Y(\Ureplace;\R^N)$ and $f:\Ureplace\times\R^N \to \R$ is jointly upper semicontinuous, it holds that
\begin{align}\label{eq:lowerfundyoung}
\begin{split}
\limsup_{j \to \infty}\int_{\Ureplace} f\left(x,\frac{d\mu_j }{dx}\right)\,dx&+\int_{\overline{\Ureplace}} f^{\#}\left(x,\frac{d\mu_j^s}{d\abs{\mu_j^s}}\right)\,d\abs{\mu_j^s} \\
&\leq \int_{\Ureplace} \inner{f(x,\cdot),\nu_x}\,dx + \int_{\overline{\Ureplace}} \inner{f^{\#}(x,\cdot),\nu^{\infty}_x}\,d\lambda_{\nu},
\end{split}
\end{align}
see~\cite[Corollary~2.10]{ADR}.
\subsection{Fractional Calculus}
Here we introduce the fractional operators and their properties, which we use throughout the paper. Firstly, for an integrable function $u \in L^1(\R^n)$, the Riesz potential $I_{\alpha}u$ of order $\alpha \in (0,n)$ is defined as
\begin{equation*}
I_{\alpha}u(x) = \frac{1}{\gamma_{n,\alpha}}\int_{\R^n} \frac{u(y)}{\abs{x-y}^{n-\alpha}}\,dy \quad \text{for $x \in \R^n$},
\end{equation*}
where $\gamma_{n,\alpha}=\pi^{n/2}2^{\alpha}\frac{\Gamma(\alpha/2)}{\Gamma((n-\alpha)/2)}$. It is well-known that the above integral is finite for a.e.~$x \in \R^n$ and $I_{\alpha}u \in L^1_{\rm loc}(\R^n)$, cf.~\cite{Stein, Mizuta}. Next up, we have the three different fractional differential operators: The Riesz fractional gradient, the fractional divergence and the fractional Laplacian. We introduce these notions for the class of bounded Lipschitz functions. Precisely, for $\alpha \in (0,1)$ and $\phi \in \Lipb(\R^n)$ the Riesz fractional gradient $\na \phi:\R^n \to \R^n$ is given by
\begin{equation}\label{eq:rieszfracgrad}
\na \phi (x) = \mu_{n,\alpha}\int_{\R^n} \frac{\phi(y)-\phi(x)}{\abs{y-x}^{n+\alpha}}\frac{y-x}{\abs{y-x}}\,dy \quad \text{for $x \in \R^n$},
\end{equation}
with $\mu_{n,\alpha}=2^{\alpha}\pi^{-n/2}\frac{\Gamma((n+\alpha+1)/2)}{\Gamma((1-\alpha)/2)}$, and the fractional Laplacian $\da{\alpha/2} \phi:\R^n \to \R$ is defined as
\begin{align*}
\da{\alpha/2}\phi(x)= \nu_{n,\alpha}\int_{\Rn}\frac{\phi(y)-\phi(x)}{\abs{y-x}^{n+\alpha}}\,dy \quad \text{for $x\in \R^n$},
\end{align*}
with $\nu_{n,\alpha}=2^{\alpha}\pi^{-n/2}\frac{\Gamma((n+\alpha)/2)}{\Gamma(-\alpha/2)}$. Both these operators are well-defined and bounded functions for $\phi \in \Lipb(\R^n)$, see~\cite[Section~2.2]{comi1}. Finally, for a vector-valued function $\phi \in \Lipb(\R^n;\R^n)$, the fractional divergence $\diva \phi:\R^n \to \R$ is the natural analogue of the fractional gradient 
\begin{equation}\label{eq:fracdivergence}
\diva \phi (x) = \mu_{n,\alpha} \int_{\R^n} \frac{\phi(y)-\phi(x)}{\abs{y-x}^{n+\alpha}}\cdot \frac{y-x}{\abs{y-x}}\,dy \quad \text{for $x \in \R^n$,}
\end{equation}
which is also a well-defined bounded function. We note that it is proven in \cite{Silhavy} that these three fractional differential operators are the unique operators that satisfy translational and rotational invariance, $\alpha$-homogeneity and a weak requirement of continuity. The fractional gradient and divergence are dual, in the sense that for all $\phi \in \Cc(\R^n)$ and $\psi \in \Cc(\R^n;\R^n)$ the integration by parts
\begin{equation}\label{eq:intbyparts}
\int_{\R^n} \phi \,\diva \psi \,dx = - \int_{\R^n} \na \phi \cdot \psi \,dx
\end{equation}
holds. For more on these differential operators, such as composition rules and extension to different orders than $\alpha \in (0,1)$, we refer to \cite{Silhavy}.

\section{Spaces of Bounded Fractional Variation}\label{sec:bv}
The spaces of bounded fractional variation were first introduced by Comi \& Stefani in the recent series of papers \cite{comi1,comi2,comi3}. We recall the definition of these spaces, which is based on the fractional divergence in \eqref{eq:fracdivergence}.
\begin{definition}
Let $\alpha \in (0,1)$. A function $u \in L^1(\R^n)$ belongs to $BV^{\alpha}(\R^n)$ if
\begin{equation}\label{eq:variation}
\sup\left\{\int_{\R^n} u \,\diva \phi \dx \,:\,\phi \in \Cc(\R^n;\R^n), \norm{\phi}_{L^{\infty}(\R^n;\R^n)} \leq 1\right\} < \infty.
\end{equation}
\end{definition}
It follows from the structure theorem \cite[Theorem 3.2]{comi1} that $u \in BV^{\alpha}(\R^n)$ if and only if $u \in L^1(\R^n)$ and there exists a (necessarily unique) finite vector-valued Radon measure $D^{\alpha}u \in \Mcal(\R^n;\R^n)$ such that
\[
\int_{\R^n} u \,\diva \phi \dx = - \int_{\R^n} \phi \cdot \,d{D^{\alpha}u} \qquad\text{for all $\phi \in \Cc(\R^n;\R^n)$}.
\]
The measure $\Da u$ is called the fractional variation measure of $u$ and it constitutes a natural extension of the Riesz fractional gradient \eqref{eq:rieszfracgrad} to the space $\bva(\R^n)$ based on the integration by parts formula \eqref{eq:intbyparts}. The space $\bva(\R^n)$ endowed with the norm
\[
\norm{u}_{\bva(\R^n)}= \norm{u}_{L^1(\R^n)} + \abs{\Da u}(\R^n)
\]
is a Banach space \cite[Corollary~3.4]{comi1}, where $\abs{\Da u}(\R^n)$ denotes the total variation of $\Da u$ on $\R^n$ and equals the left-hand side of \eqref{eq:variation}. One can also decompose
\[
D^{\alpha}u = \na u \dx + D^{\alpha}_su,
\]
where $\na u \in L^1(\R^n;\R^n)$ is the absolutely continuous part of $\Da u$ with respect to the Lebesgue measure and $\Da_s u \in \Mcal(\R^n;\R^n)$ is the singular part. We write $\bva(\R^n;\R^m)$ for the vector-valued analogue with matrix-valued fractional variation. We also introduce the fractional Sobolev space with exponent $p=1$
\[
S^{\alpha,1}(\R^n)=\{u \in \bva(\R^n) \,:\, \Da_s u =0\},
\]
which consists of those $\bva$-functions with an absolutely continuous fractional variation. In fact, these are the functions that have a well-defined weak fractional gradient $\na u \in L^1(\R^n;\R^n)$, see~\cite{comi1,comi2,comi3,KrS22} for more on these fractional Sobolev spaces.

As in \cite{KrS22}, the main tool we use to prove the lower semicontinuity and relaxation result is a method to transform the fractional gradient into the classical gradient and back. It relies on the Riesz potential and fractional Laplacian and is proven in the $BV$-framework in \cite[Lemma~3.28]{comi1}.
\begin{proposition}\label{prop:connection}
Let $\alpha \in (0,1)$, then the following holds:
\begin{itemize}
\item[$(i)$] For $u \in \bva(\R^n)$ one has that $v=I_{1-\alpha}u \in BV_{\rm loc}(\R^n)$ with $Dv=\Da u$ in $\Mcal(\R^n;\R^n)$.
\item[$(ii)$] For $v \in BV(\R^n)$ one has that $u=\dac v \in \bva(\R^n)$ with $\Da u = Dv$ in $\Mcal(\R^n;\R^n)$ and
\[
\norm{u}_{\bva(\R^n)} \leq c_{n,\alpha} \norm{v}_{BV(\R^n)}.
\]
\end{itemize}
\end{proposition}

Another ingredient we need is the Leibniz rule for the fractional variation, in order to employ localization techniques. We define for $\phi \in \Ccr$ and $\psi \in \Lip_c(\R^n)$ the operator
\[
\nanl(\phi,\psi)(x) = \mu_{n,\alpha}\int_{\R^n}\frac{(y-x)(\phi(y)-\phi(x))(\psi(y)-\psi(x))}{\abs{y-x}^{n+\alpha+1}}\dy, \quad \text{for $x \in \R^n$},
\]
which can be continuously extended to $\phi \in L^1(\R^n)$. The following Leibniz rule for $\bva$-functions is from \cite[Lemma~5.6]{Comiprecise}, see also~\cite{Comileibniz} for more general Leibniz rules.
\begin{lemma}\label{le:leibniz}
Let $\alpha \in (0,1)$, $\psi \in \Lip_c(\R^n)$ and $u \in \bva(\R^n)$. Then, $\psi u \in \bva(\R^n)$ with
\begin{equation*}
\Da(\psi u) = \psi \Da u+ (u\na \psi + \nanl(u,\psi))\dx
\end{equation*}
and there is a constant $C=C(n,\alpha)>0$ such that
\begin{equation}\label{eq:leibnizbound}
\norm{u\na \psi + \nanl(u,\psi)}_{L^1(\R^n;\R^n)}\leq C\norm{\psi}_{L^{\infty}(\R^n)}^{1-\alpha}\Lip(\psi)^{\alpha}\norm{u}_{L^1(\R^n)}.
\end{equation}
\end{lemma}
\begin{remark}\label{rem:locality}
Even though the fractional variation is a nonlocal object, the above Leibniz rule implies that the singular part of the fractional variation behaves locally. Indeed, if we have $u,v \in \bva(\R^n)$ with $u=v$ in an open set $U \subset \R^n$, we find for all $\chi \in C_c^{\infty}(U)$ (extended to $\R^n$ as zero) that
\[
\chi \Da_s u = \Da_s(\chi u) = \Da_s (\chi v) = \chi \Da_s v. \qedhere
\]
\end{remark}
For our minimization problems, we restrict to functions satisfying a complementary-value condition, which is a nonlocal counterpart of the common Dirichlet boundary conditions. For $\Omega \subset \R^n$ open and bounded we define
\[
\bva_0(\Omega) = \{ u \in \bva(\R^n) \,:\, u = 0 \ \text{a.e.~in $\Omega^c$}\},
\]
and for $g \in S^{\alpha,1}(\R^n)$
\[
\bvag(\Omega) =g+\bva_0(\R^n).
\]
Here, we take $g \in S^{\alpha,1}(\R^n)$ since our initial motivation comes from studying linear growth functionals on the fractional Sobolev space. With this in mind, we also introduce the spaces $S^{\alpha,1}_0(\Omega)$ and $S^{\alpha,1}_g(\Omega)$ in a similar way as above. For $u \in \bvag(\Omega)$, it follows that the singular part $\Da_s u$ has support inside $\overline{\Omega}$, because of the local behavior of the singular part of the fractional variation (cf.~Remark~\ref{rem:locality}).

A key reason to consider the fractional $BV$-spaces as extension of the fractional Sobolev spaces, is the property that bounded sequences have convergent subsequences in $\bvag(\Omega)$ in an appropriate sense. We say that $(u_j)_j \subset \bvag(\Omega)$ converges weak* to $u \in \bvag(\Omega)$ if
\begin{equation*}
u_j \to u \ \text{in $L^1(\R^n)$} \quad \text{and} \quad \Da u_j \starto \Da u \ \text{in $\Mcal(\R^n;\R^n)$ as $j \to \infty$}.
\end{equation*}
A direct application of the Banach-Alaoglu theorem and the compactness result \cite[Theorem~3.16]{comi1} shows that bounded sequences in $\bvag(\Omega)$ admit weak* convergent subsequences.
Moreover, we have the following result stating that weak* convergence improves to strong $L^1$-convergence outside $\Omega$. We omit the proof as it is almost identical to that of \cite[Lemma~2.12]{KrS22}.
\begin{lemma}\label{le:strongoutside}
Let $\alpha \in (0,1)$, $\Omega$ be open and bounded and $g \in S^{\alpha,1}(\R^n)$. If $(u_j)_j \subset \bvag(\Omega)$ converges weak* to $u$ in $\bvag(\Omega)$, then for every open $\Omega' \Supset \Omega$ we find
\begin{equation*}
\na u_j \to \na u \ \text{in $L^1((\Omega')^c;\R^n)$ as $j \to \infty$}.
\end{equation*}
\end{lemma}

We also need an improved version of the Poincar\'{e} inequality for fractional $BV$-functions in \cite{comi2}, which only requires a bound on the fractional variation on some open and bounded set as opposed to the whole space $\R^n$. This allows us to consider interesting integrands with slightly weaker coercivity properties, such as the area-integrand in Example~\ref{ex:area}.
\begin{proposition}\label{prop:poincare}
Let $\alpha \in (0,1)$ and $\Omega$ be open and bounded, then there exists an open and bounded set $\Omega' \Supset \Omega$ and a constant $C=C(\Omega,n,\alpha)>0$ such that
\[
\norm{u}_{\bva(\R^n)} \leq C \abs{\Da u}(\Omega'),
\]
for all $u \in \bva_0(\Omega)$.
\end{proposition}
\begin{proof}
Define for $r>0$ the function
\[
\chi:\R^n \to \R, \quad \chi(x) = \max\{1-r\,d(x,\Omega),0\},
\]
with $d(x,\Omega)$ the distance from $x$ to $\Omega$. Then, we have $\chi \in \Lip_c(\R^n)$, $\Lip(\chi) \leq r$, $\chi \equiv 1$ on $\Omega$ and 
\[
\supp(\chi) = \Omega_r:=\{x \in \R^n \,:\, d(x,\Omega) \leq 1/r\}.
\]
We deduce that $u=\chi u$ and conclude from the Leibniz rule (Lemma~\ref{le:leibniz}) that
\[
\Da u= \Da (\chi u)=\chi \Da u + (u\na \chi + \nanl(u,\chi))\dx.
\]
Therefore, we find by \eqref{eq:leibnizbound} that
\begin{equation}\label{eq:decaybound}
\norm{\na u}_{L^1(\Omega^c_r;\R^n)} = \norm{u\na \chi + \nanl(u,\chi)}_{L^1(\Omega^c_r;\R^n)} \leq Cr^{\alpha}\norm{u}_{L^1(\R^n)}.
\end{equation}
Now using H\"{o}lder's inequality on the scale of Lorentz spaces, see~\cite[Chapter~1.4]{Grafakos} for an introduction on Lorentz spaces, in combination with the weak Gagliardo-Nirenberg-Sobolev inequality from \cite[Theorem~3.8]{comi2} yields
\begin{equation}\label{eq:gnsineq}
\norm{u}_{L^1(\Omega)} \leq \norm{\mathbbm{1}_{\Omega}}_{L^{\frac{n}{\alpha},1}(\R^n)}\norm{u}_{L^{\frac{n}{n-\alpha},\infty}(\R^n)} \leq \frac{n\abs{\Omega}^{\alpha/n}}{\alpha}c_{n,\alpha}\abs{\Da u}(\R^n)=c_{n,\alpha,\Omega}\abs{\Da u}(\R^n),
\end{equation}
for all $u \in \bva_0(\Omega)$. If we choose $r>0$ such that $Cr^{\alpha} \leq (2c_{n,\alpha,\Omega})^{-1}$ we obtain from \eqref{eq:decaybound} and \eqref{eq:gnsineq} that
\begin{align*}
\norm{u}_{L^{1}(\Omega)} &\leq c_{n,\alpha,\Omega}\left(\abs{\Da u}(\Omega_r)+\abs{\Da u}(\Omega_r^c)\right)\\
&\leq c_{n,\alpha,\Omega} \left(\abs{\Da u}(\Omega_r)+\frac{1}{2c_{n,\alpha,\Omega}}\norm{u}_{L^1(\Omega)}\right),
\end{align*}
which, after rewriting, becomes
\begin{equation}\label{eq:strongpoincare}
\norm{u}_{L^1(\Omega)} \leq 2c_{n,\alpha,\Omega}\abs{\Da u}(\Omega_r).
\end{equation}
Therefore, we obtain
\begin{align*}
\norm{u}_{\bva(\R^n)} &= \norm{u}_{L^1(\Omega)}+\abs{\Da u}(\Omega_r)+\abs{\Da u}(\Omega_r^c)\\
&\leq \left(1+\frac{1}{2c_{n,\alpha,\Omega}}\right)\norm{u}_{L^1(\Omega)} + \abs{\Da u}(\Omega_r)\\
&\leq (2c_{n,\alpha,\Omega}+2) \abs{\Da u}(\Omega_r),
\end{align*}
which proves the result with any open set $\Omega' \Supset \Omega_r$.
\end{proof}

Next, to extend the linear growth functionals from $S^{\alpha,1}_g(\Omega;\R^m)$ to $\bvagom$ we need to be able to approximate functions in $\bvagom$ with functions in $S^{\alpha,1}_g(\Omega;\R^m)$ in a strong enough sense to also have convegence of the functional values. However, since $\Sonegom$ is closed with respect to the $\bva$-norm, we have to consider a convergence notion that is also weaker than the one induced by the norm. The relevant notion here is a type of area-strict convergence, which is in between norm convergence and weak* convergence. Like in \cite{KrR10b}, we define the area-functional for $\mu \in \Mcal(\R^n;\R^N)$ and $U \subset \R^n$ Borel measurable as
\[
\inner{\mu}(U):=\int_{U}\sqrt{1+\abs*{\frac{d\mu }{dx}}^2}\dx+\abs{\mu^s}(U),
\]
with $\mu^s$ the singular part of $\mu$ with respect to the Lebesgue measure. We also write $\inner{A}:=\sqrt{1+\abs{A}^2}$ for $A \in \Rmn$.
\begin{definition}[area-strict convergence]\label{def:areastrict} 
We say that a sequence $(u_j)_j \subset \bvagom$ converges area-strictly to $u \in \bvagom$ if $u_j \to u \ \text{in $L^1(\R^n;\R^m)$}$,
\[
\inner{\Da u_j}(\overline{\Omega}) \to \inner{\Da u}(\overline{\Omega}) \quad \text{and} \quad \na u_j \to \na u \ \text{in $L^1(\Omega^c;\Rmn)$} \ \text{as $j \to \infty$}.
\]

\end{definition}
\begin{remark}\label{rem:areastrict}
The key property of area-strict convergence is that when restricted to $\overline{\Omega}$, the sequence $(\Da u_j)_j \subset \Mcal(\overline{\Omega};\Rmn)$ generates the elementary Young measure $\delta[\Da u]$ (cf.~\eqref{eq:elementary} and \cite[Proposition~12.4]{Rindler}). The convergence $\na u_j \to \na u \ \text{in $L^1(\Omega^c;\Rmn)$}$ also excludes any concentration effects happening outside $\Omega$, which are in general not ruled out by Lemma~\ref{le:strongoutside}. 
\end{remark} 
We now prove a density result with respect to the area-strict convergence, which plays a key role in the construction of a recovery sequence when extending the linear growth functionals. The proof exploits the fractional Leibniz rule and invariance properties of the fractional variation to incorporate the partition of unity and mollification techniques from the classical case (as in e.g.~\cite[Lemma~11.1]{Rindler}). Note that we implicitly assume that functions in $C_c^{\infty}(\Omega;\R^m)$ are extended to $\R^n$ as zero.
\begin{theorem}\label{th:density}
Let $\alpha \in (0,1)$, $\Omega \subset\R^n$ be a bounded Lipschitz domain and $g \in \Sone(\R^n;\R^m)$. For every $u \in \bvagom$ there exists a sequence $(u_j)_j \subset g+C_c^{\infty}(\Omega;\R^m)$ such that
\[
u_j \to u \ \text{area-strictly in $\bvagom$}.
\]
\end{theorem}
\begin{proof}
\textit{Step 1: Shrinking the support.} We show that for every $\epsilon >0$, we can find a $v \in \bvagom$ such that $\supp(v-g) \Subset \Omega$,
\begin{equation}\label{eq:steponegoal}
\norm{u-v}_{L^1(\R^n;\R^m)}+\norm{\na u - \na v}_{L^1(\Omega^c;\Rmn)}\leq \epsilon \quad\text{and}\quad \inner{\Da v}(\overline{\Omega}) \leq \inner{\Da u}(\overline{\Omega})+\epsilon.
\end{equation}
To this aim, we take a representative of $u$ that is identical to $g$ in $\Omega^c$ and set $u_0:=u-g$. Then, since $\Omega$ is a Lipschitz domain, we find a partition of unity $\chi_0,\chi_1,\cdots, \chi_N \subset \Ccr$ and vectors $\zeta_1,\cdots,\zeta_N \subset \R^n$ such that
\begin{equation}\label{eq:partition}
\sum_{i=0}^N \chi_i =1 \ \text{on $\Omega$}, \quad\chi_0 \in C_c^{\infty}(\Omega), \quad \text{and} \quad \supp(\tau_{\lambda \zeta_i}(\chi u_0)) \Subset \Omega,
\end{equation}
for all $\lambda>0$ small enough, where $\tau_{\zeta}(w)(x):=w(x+\zeta)$ denotes translation by $\zeta \in \R^n$. In view of Lemma~\ref{le:leibniz} we can define the function
\[
v=g+ \chi _0u_0 + \sum_{i=1}^{N} \tau_{\lambda \zeta}(\chi_i u_0) \in \bvagom,
\]
which satisfies $\supp(v-g) \Subset \Omega$ due to \eqref{eq:partition}. Using the first identity from \eqref{eq:partition}, we have that
\[
\norm{u-v}_{L^1(\R^n;\R^m)}\leq \sum_{i=1}^{N}\norm{\chi_i u_0 - \tau_{\lambda\zeta_i}(\chi_i u_0)}_{L^1(\R^n;\R^m)} \leq \epsilon/2,
\]
for $\lambda$ small enough given the continuity of translation on $L^1(\R^n;\R^m)$. Moreover, we have by the translation invariance of $\na$ that
\[
\na v = \na u +\sum_{i=1}^N \tau_{\lambda \zeta_i}(\na (\chi_i u_0)) -\na(\chi_i u_0)
\]
so that the continuity of translation on $L^1(\R^n;\Rmn)$ again yields
\[
\norm{\na v-\na u}_{L^1(\R^n;\Rmn)}= \sum_{i=1}^N \norm{\tau_{\lambda \zeta_i}(\na (\chi_i u_0)) -\na(\chi_i u_0)}_{L^1(\R^n;\Rmn)} \leq \epsilon/2
\]
for $\lambda$ small enough. This shows the first part of \eqref{eq:steponegoal} and at the same time that
\begin{align}\label{eq:absolute}
\int_{\Omega}\inner{\na v}\,dx &\leq \int_{\Omega}\inner{\na u}\,dx + \norm{\na v-\na u}_{L^1(\Omega;\Rmn)}
\leq \int_{\Omega}\inner{\na u}\,dx + \epsilon/2,
\end{align}
where we have exploited the 1-Lipschitz continuity of $A \mapsto \inner{A}$. Finally, for the singular part we note that
\[
\Da_s v= \chi_0\Da_s u_0+\sum_{i=1}^{N} \tau_{\lambda \zeta_i}(\chi_i \Da_s u_0) =\chi_0\Da_s u+\sum_{i=1}^{N} \tau_{\lambda \zeta_i}(\chi_i \Da_s u)
\]
in virtue of Lemma~\ref{le:leibniz}. Hence, it follows with \eqref{eq:partition} that
\begin{align*}
\abs{\Da_s v}(\overline{\Omega}) &\leq \int_{\overline{\Omega}} \chi_0\,d\abs{\Da_s u}+\sum_{i=1}^N \int_{\tau_{-\lambda \zeta_i}(\overline{\Omega})} \chi_i \,d\abs{\Da_s u}\\
&\leq \sum_{i=0}^{N} \int_{\R^n} \chi_i \,d\abs{\Da_s u}=\abs{\Da_s u}(\overline{\Omega}),
\end{align*} 
which proves the second part of \eqref{eq:steponegoal} in combination with \eqref{eq:absolute}. \smallskip

\textit{Step 2: Mollification.} Let $v$ be as in Step 1, then we show that there is a $w \in g+C_c^{\infty}(\Omega;\R^m)$ such that
\begin{equation}\label{eq:steptwogoal}
\norm{v-w}_{L^1(\R^n;\R^m)}+\norm{\na v - \na w}_{L^1(\Omega^c;\Rmn)}\leq \epsilon \quad\text{and}\quad \inner{\Da w}(\overline{\Omega}) \leq \inner{\Da v}(\overline{\Omega})+\epsilon.
\end{equation}
Let $\eta_{\delta} \in C_c^{\infty}(B_\delta(0))$ for $\delta>0$ be a standard mollifier and choose $\delta$ small enough such that
\[
\eta_{\delta} * (v-g) \in C_c^{\infty}(\Omega;\R^m),
\]
which is possible since $\supp(v-g)\Subset \Omega$. Setting $w=g+\eta_{\delta} * (v-g)$, standard properties of mollification show that
\[
\norm{v-w}_{L^1(\R^n;\R^m)} \leq \epsilon/2,
\]
for $\delta$ small enough. Furthermore, by \cite[Lemma~3.5]{comi1} we have
\[
\na w = \na g+\eta_{\delta}*\na (v-g)+\eta_{\delta}*\Da_s v.
\]
In particular, since $\eta_{\delta}*\Da_s v$ has support inside $\Omega$ we have
\[
\norm{\na v - \na w}_{L^1(\Omega^c;\Rmn)} = \norm{\na (v-g) - \eta_{\delta}*\na (v-g)}_{L^1(\Omega^c;\Rmn)} \leq \epsilon/2,
\]
for small $\delta$, thus proving the first part of \eqref{eq:steptwogoal}. Furthermore,
\begin{align*}
\inner{\Da w}(\overline{\Omega})&=\int_{\Omega}\inner{\na g+\eta_{\delta}*\na (v-g)+\eta_{\delta}*\Da_s v}\dx\\
&\leq\int_{\Omega}\inner{\na g+\eta_{\delta}*\na(v-g)}\dx+\int_{\Omega}\abs{\eta_{\delta}*\Da_s v}\dx\\
&\leq \int_{\Omega}\inner{\na g+\eta_{\delta}*\na(v-g)}\dx+\abs{\Da_s v}(\overline{\Omega})\\
&\leq \int_{\Omega}\inner{\na v}\dx+\epsilon+  \abs{\Da_s v}(\overline{\Omega})=\inner{\Da v}(\overline{\Omega})+\epsilon,
\end{align*}
where in the last line we utilize Lebesgue's dominated convergence for small enough $\delta$, recalling the fact that $\eta_{\delta}*\na (v-g) \to \na (v-g)$ in $L^1(\R^n;\Rmn)$. This yields \eqref{eq:steptwogoal}.\smallskip

\textit{Step 3: Conclusion.} By combining Step 1 and 2 we may find a sequence $(u_j)_j \subset g+C_c^{\infty}(\Omega;\R^m)$ such that
\[
u_j \to u \ \text{in $L^1(\R^n;\R^m)$}, \quad \na u_j \to \na u \ \text{in $L^1(\Omega^c;\Rmn)$ as $j \to \infty$}
\]
and
\[
\limsup_{j \to \infty}\, \inner{\Da u_j}(\overline{\Omega}) \leq \inner{\Da u}(\overline{\Omega}).
\]
In view of this bound we have that $u_j \starto u$ in $\bvagom$. Therefore, we may use the weak* lower semicontinuity of the area-functional on $\Mcal(\Omega';\Rmn)$ for some $\Omega' \Supset \Omega$ open and bounded, which follows from the convexity of $A \mapsto \inner{A}$, and Lebesgue's dominated convergence theorem to conclude that
\begin{align*}
\liminf_{j \to \infty} \inner{\Da u_j}(\overline{\Omega})&=\liminf_{j \to \infty} \inner{\Da u_j}(\Omega')-\lim_{j \to \infty} \int_{\Omega'\setminus \Omega} \inner{\na u_j}\,dx \\
& \geq \inner{\Da u}(\Omega')-\int_{\Omega'\setminus \Omega} \inner{\na u}\,dx=\inner{\Da u}(\overline{\Omega}),
\end{align*}
which finishes the proof.
\end{proof}

\section{Lower Semicontinuity}
In this section we characterize the weak* lower semicontinuity of functionals as in \eqref{eq:extendedfunctional}, which is interesting in its own right and is used in the proof of the main relaxation result in Section~\ref{sec:relaxation}. 
Recall that a continuous function $h:\Rmn \to \R$ is called quasiconvex if
\[
h(A) \leq \int_{(0,1)^n}h(A+ \nabla \phi(y))\,dy \quad \text{for all $A \in \Rmn$ and $\phi \in W^{1,\infty}_0((0,1)^n;\R^m)$},
\]
see \cite{Mor, Dacorogna}. We prove the following statement, whose proof relies on the connection between the classical and fractional variation and the theory of generalized Young measures. We note that even though $f^{\infty}(x,A)$ is only assumed to exist for $x \in \overline{\Omega}$, we do allow the sequences $x'\to x$ from \eqref{eq:recessionfrac} to approach from outside $\overline{\Omega}$. 
\begin{theorem}\label{th:sufficiency}
Let $\alpha \in (0,1)$, $\Omega \subset \R^n$ be open and bounded with $\abs{\partial \Omega}=0$, $g \in S^{\alpha,1}(\R^n;\R^m)$ and $f:\R^n \times \Rmn \to \R$ a Carath\'{e}odory integrand that satisfies \eqref{eq:growth}.
If 
\[
f^{\infty}(x,A) \ \text{exists for all $(x,A) \in \overline{\Omega}\times \Rmn$},
\]
then the functional
\[
\overline{\F}_{\alpha}(u)=\int_{\R^n} f(x,\na u)\dx+\int_{\overline{\Omega}}f^{\infty}\left(x,\frac{d\Da_s u}{\abs{d\Da_s u}}\right)\,d\abs{\Da_su} \quad \text{for $u \in \bvagom$},
\]
is weak* lower semicontinuous if and only if $f(x,\cdot)$ is quasiconvex for a.e.~$x \in \Omega$.
\end{theorem}
\begin{proof}
\textit{Step 1: Necessity.} The weak* lower semicontinuity of $\overline{\Fcal}_{\alpha}$ implies, in particular, that
\[
\F_{\alpha}(u) =\int_{\R^n} f(x,\na u(x))\dx \quad \text{for $u \in S^{\alpha,1}_g(\Omega;\R^m)$},
\]
is weakly lower semicontinuous on $S^{\alpha,1}_g(\Omega;\R^m)$. A simple adaptation of \cite[Theorem~4.5]{KrS22} to the case $p=1$ yields that $f(x,\cdot)$ is quasiconvex for a.e.~$x \in \Omega$. \smallskip

\textit{Step 2: Sufficiency.} Let $u_j \starto u$ in $\bvagom$ and fix $\Omega'\Supset \Omega$ open and bounded. By Proposition~\ref{prop:connection}\,$(i)$ we can find a sequence $(v_j)_j \subset BV(\Omega';\R^m)$ and $v \in BV(\Omega';\R^m)$ such that
\begin{equation}\label{eq:gradients}
Dv_j=\Da u_j \ \text{on $\Omega'$ for $j \in \N$} \quad\text{and}\quad Dv=\Da u \ \text{on $\Omega'$};
\end{equation}
we can also ensure that $v_j \starto v$ in $BV(\Omega';\R^m)$ by continuity properties of the Riesz potential, see~\cite[Theorem~2.1\,(i)]{Mizuta}. Up to a non-relabeled subsequence, $(D v_j)_j \subset \Mcal(\overline{\Omega'};\Rmn)$ generates a $BV$-Young measure $\nu \in Y(\Omega';\Rmn)$ on $\Omega'$. Before we proceed, we can redefine $f$, similarly to \cite[Proof of Theorem~12.25]{Rindler}, such that its recession function is defined in a larger region. Indeed, by definition of the strong recession function $f^{\infty}:\overline{\Omega} \times \Rmn \to \R$, it is automatically jointly continuous, so that we can continuously extend it to $\overline{\Omega'} \times \Rmn$. If we set $f':\overline{\Omega'} \times \Rmn \to \R$ equal to $f$ on $\overline{\Omega} \times \Rmn$ and $f^{\infty}$ on $(\overline{\Omega'} \setminus \overline{\Omega})  \times \Rmn$, then $f'$ is a Carath\'{e}odory integrand with a well-defined strong recession function on $\overline{\Omega'} \times \Rmn$. Now, applying Theorem~\ref{th:fundyoung} to $f'$ and $\Omega'$ yields
\begin{align}\label{eq:yinside}
\begin{split}
\liminf_{j \to \infty}& \int_{\Omega'} f(x,\na u_j)\dx+\int_{\overline{\Omega}}f^{\infty}\left(x,\frac{d\Da_s u_j}{\abs{d\Da_s u_j}}\right)\,d\abs{\Da_su_j} \\
&\geq \liminf_{j \to \infty} \int_{\Omega'} f'(x,\na u_j)\dx+\int_{\overline{\Omega'}}(f')^{\infty}\left(x,\frac{d\Da_s u_j}{\abs{d\Da_s u_j}}\right)\,d\abs{\Da_su_j}\\
&\qquad \quad -\sup_{j \in \N}\int_{\Omega' \setminus \Omega} \abs{(f-f')(x,\na u_j)}\,dx\\
&\geq \int_{\Omega'}\inner{f'(x,\cdot),\nu_x}\dx + \int_{\overline{\Omega'}}\inner{(f')^{\infty}(x,\cdot),\nu^{\infty}_x}\,d\lambda^{\nu}-\sup_{j \in \N}\int_{\Omega' \setminus \Omega} \abs{(f-f')(x,\na u_j)}\,dx.
\end{split}
\end{align}
Due to the strong convergence $\nabla v_j =\na u_j \to \na u = \nabla v$ in sets away from $\Omega$ (Lemma~\ref{le:strongoutside}), we also find that the support of the concentration measure $\lambda_{\nu}$ is contained inside $\overline{\Omega}$ and $\nu_x=\delta_{\na u(x)}$ for a.e.~$x \in \Omega'\setminus \Omega$; that is, we find
\begin{equation}\label{eq:youtside}
\int_{\Omega' \setminus \Omega} \inner{f'(x,\cdot),\nu_x}\dx = \int_{\Omega' \setminus \Omega} f'(x,\na u)\,dx.
\end{equation}
Furthermore, since $\nu$ is a $BV$-Young measure generated by $(D v_j)_j$, we may argue as in \cite[Theorem~10]{KrR10a} and \cite[Theorem~12.25]{Rindler} using the generalized Jensen's inequalities from \cite[Theorem~9]{KrR10a} in combination with the quasiconvexity, continuity and linear growth of $f(x,\cdot)$ for a.e.~$x \in \Omega$ and of $f^{\infty}(x,\cdot)$ for all $x \in \overline{\Omega}$ (by continuity of $f^{\infty}$) to conclude
\begin{equation}\label{eq:yyinside}
\begin{split}
\int_{\Omega}\inner{f(x,\cdot),\nu_x}\dx + \int_{\overline{\Omega}}\inner{f^{\infty}(x,\cdot),\nu^{\infty}_x}\,d\lambda_{\nu} &\geq \int_{\Omega} f(x,\nabla v)\dx+\int_{\overline{\Omega}}f^{\infty}\left(x,\frac{dD_s v}{\abs{dD_s v}}\right)\,d\abs{D_s v}\\
&=\int_{\Omega} f(x,\na u)\dx+\int_{\overline{\Omega}}f^{\infty}\left(x,\frac{d\Da_s u}{\abs{d\Da_s u}}\right)\,d\abs{\Da_s u},
\end{split}
\end{equation}
with the last equality exploiting \eqref{eq:gradients}. Additionally, since $\na u_j \to \na u$ strongly in $L^1((\Omega')^c;\Rmn)$ by Lemma~\ref{le:strongoutside}, the growth bound \eqref{eq:growth} and Lebesgue's dominated convergence theorem yields
\begin{equation*}
\lim_{j \to \infty} \int_{(\Omega')^c} f(x,\na u_j)\dx = \int_{(\Omega')^c} f(x,\na u)\dx.
\end{equation*}
Combining this with \eqref{eq:yinside}, \eqref{eq:youtside} and \eqref{eq:yyinside} results in
\begin{align}\label{eq:almost}
\begin{split}
\liminf_{j \to \infty} \overline{\Fcal}_{\alpha}(u_j) &\geq  \int_{\Omega'} f'(x,\na u)\dx+\int_{(\Omega')^c}f(x,\na u)\,dx+\int_{\overline{\Omega}}f^{\infty}\left(x,\frac{d\Da_s u}{\abs{d\Da_s u}}\right)\,d\abs{\Da_su}\\
& \qquad \quad -\sup_{j \in \N}\int_{\Omega' \setminus \Omega} \abs{(f-f')(x,\na u_j)}\,dx.
\end{split}
\end{align}
Finally, since 
\[
(f-f')^{\infty}(x,A)=0 \quad \text{for all $x \in \overline{\Omega}$ and $A \in \Rmn$,}
\]
we may choose $\Omega'$ potentially smaller and find a $R>0$ such that
\[
\abs{(f-f')(x,A)} \leq \epsilon \abs{A} \quad \text{for all $x \in \Omega'$ and $A \in \Rmn$ with $\abs{A} \geq R$},
\]
for any given $\epsilon >0$. With $C:=\sup_{j} \norm{\na u_j}_{L^1(\R^n;\Rmn)} < \infty$, we obtain
\[
\sup_{j \in \N}\int_{\Omega' \setminus \Omega} \abs{(f-f')(x,\na u_j)}\,dx \leq (MR+\norm{a}_{L^{\infty}(\R^n)})\abs{\Omega'\setminus \Omega}+\epsilon C,
\]
so that we can deduce the result by first letting $\Omega' \downarrow \Omega$, given that $f=f'$ on $\Omega \times \Rmn$, and secondly letting $\epsilon \downarrow 0$ in \eqref{eq:almost}.
\end{proof}

In order to get the existence of minimizers, we also impose the coercivity bound \eqref{eq:coercivity} and utilize the improved Poincar\'{e} inequality from Proposition~\ref{prop:poincare}.
\begin{corollary}\label{cor:existence}
Let $\alpha \in (0,1)$, $\Omega \subset \R^n$ be open and bounded with $\abs{\partial \Omega}=0$, $g \in S^{\alpha,1}(\R^n;\R^m)$ and $f:\R^n \times \Rmn \to \R$ a Carath\'{e}odory integrand that satisfies \eqref{eq:growth} and \eqref{eq:coercivity}. If 
\[
f^{\infty}(x,A) \ \text{exists for all $(x,A) \in \overline{\Omega}\times \Rmn$},
\]
and $f(x,\cdot)$ is quasiconvex for a.e.~$x \in \Omega$, then
\[
\overline{\F}_{\alpha}(u)=\int_{\R^n} f(x,\na u)\dx+\int_{\overline{\Omega}}f^{\infty}\left(x,\frac{d\Da_s u}{\abs{d\Da_s u}}\right)\,d\abs{\Da_su} \quad \text{for $u \in \bvagom$},
\]
admits a minimizer on $\bvagom$.
\end{corollary}
\begin{proof}
We fix $\Omega' \Supset \Omega$ large enough as in Proposition~\ref{prop:poincare} and such that $M\abs{\Da v}((\Omega')^c) \leq \frac{\mu }{2}\abs{\Da v}(\Omega')$ for all $v \in BV^{\alpha}_0(\Omega;\R^m)$, which is possible by \eqref{eq:decaybound} and \eqref{eq:strongpoincare}. Now using the coercivity condition of $f$ on $\Omega'$, the growth bound on $(\Omega')^c$ and Proposition~\ref{prop:poincare}, we find for all $u \in \bvagom$ that
\begin{align*}
\overline{\Fcal}_{\alpha}(u) &\geq \mu \abs{D^{\alpha}u}(\Omega')-M\abs{\Da u}((\Omega')^c)-C'\\
& \geq \mu \abs{\Da (u-g)}(\Omega')-M\abs{\Da (u-g)}((\Omega')^c)-C''\\
&\geq \frac{\mu}{2C}\norm{u-g}_{BV^{\alpha}(\R^n;\R^m)}-C''.
\end{align*}
Hence, a standard argument using the direct method and the weak* lower semicontinuity from Theorem~\ref{th:sufficiency} finishes the proof.
\end{proof}
\begin{example}\label{ex:area}
An example integrand that satisfies all the hypotheses of Corollary~\ref{cor:existence} is
\[
f:\R^n \times \R^{m \times n} \to [0,\infty), \quad f(x,A) = \sqrt{1+\abs{A}^2}-1,
\]
since $f$ is convex, $f^{\infty}(x,A)=\abs{A}$ and
\[
\abs{A}-1 \leq f(x,A) \leq \abs{A} \quad \text{ for all $x \in \R^n$ and $A \in \Rmn$}.
\]
Hence, the following type of fractional area-functional
\[
\overline{\Fcal}_{\alpha}(u) = \int_{\R^n} \sqrt{1+\abs{\na u}^2}-1\,dx + \abs{\Da_su}(\overline{\Omega}),
\]
is weak* lower semicontinuous on $\bvagom$ and admits a minimizer.
\end{example}

\section{Relaxation}\label{sec:relaxation}
We are now in the position to give the proof of the main result. For a Carath\'{e}odory integrand $f:\R^n \times \Rmn \to \R$ that satisfies the bounds \eqref{eq:growth} and \eqref{eq:coercivity}, it follows from \cite[Proposition~9.5]{Dacorogna} that
\[
f^{\rm qc}(x,A) = \inf\left\{\int_{(0,1)^n}f(x,A+\nabla \phi(y))\,dy \,:\, \phi \in W^{1,\infty}_0((0,1)^n;\R^m)\right\}
\]
for $(x,A) \in \R^n \times \Rmn$, is a Carath\'{e}odory integrand and from \cite[Theorem~6.9]{Dacorogna} that the function $f^{\rm qc}(x,\cdot)$ is the largest quasiconvex function below $f(x,\cdot)$. Note also that $f^{\rm qc}$ still satisfies \eqref{eq:growth} and \eqref{eq:coercivity} for $x \in \overline{\Omega}$ since $f^{\rm qc} \leq f$ and the lower bound in \eqref{eq:coercivity} is quasiconvex in the second argument.

\begin{proof}[Proof of Theorem~\ref{th:main}]
Denote the functional on the right-hand side of \eqref{eq:relaxation} by $\overline{\F}_{\alpha}$. We split up the proof into the lower and upper bound. \smallskip

\textit{Step 1: Lower bound.} Let $(u_j)_j \subset \Sonegom$ with $u_j \starto u$ in $\bvagom$ as $j \to \infty$, then we can completely follow the proof of Theorem~\ref{th:sufficiency} without using the generalized Jensen's inequalities to conclude (up to a non-relabeled subsequence) that
\[
\liminf_{j \to \infty} \Fcal_{\alpha}(u_j) \geq \int_{\Omega}\inner{f(x,\cdot),\nu_x}\dx + \int_{\overline{\Omega}}\inner{f^{\infty}(x,\cdot),\nu^{\infty}_x}\,d\lambda_{\nu}+\int_{\Omega^c}f(x,\na u)\,dx,
\]
with $\nu$ the generalized Young measure generated by the sequence $(\na u_j)_j$ (on some domain containing $\Omega$). Using the bounds $f \geq f^{\rm qc}$ and $f^{\infty} \geq (f^{\rm qc})^{\#}$, we can now proceed as in Theorem~\ref{th:sufficiency} by using the Jensen's inequalities for the quasiconvexification $f^{\rm qc}$, to obtain the lower bound. Here, we make crucial use of the second part of \eqref{eq:technicalcond}, since the Jensen's inequalities for upper recession functions in \cite[Theorem~9]{KrR10a} can only be directly applied in the $x$-independent case.\smallskip

\textit{Step 2: Upper bound.} We first show that we can restrict to the case $u \in g+ C_c^{\infty}(\Omega;\R^m)$ for the upper bound. To this aim, we take $u \in \bvagom$ and a sequence $(u_j)_j \subset g + C_c^{\infty}(\Omega;\R^m)$ which converges area-strictly to $u$, possible by Theorem~\ref{th:density}. Hence, Lebesgue's dominated convergence theorem and the growth bound \eqref{eq:growth} yields
\begin{equation}\label{eq:compconv}
\lim_{j \to \infty} \int_{\Omega^c}f(x,\na u_j)\dx=\int_{\Omega^c}f(x,\na u)\dx. 
\end{equation}
Next, if we denote by $g:\Omega \times \Rmn \to \R$ the (jointly) upper semicontinuous envelope of $f^{\rm qc}$ restricted to $\Omega \times \Rmn$, then it is not hard to verify that $g^{\#} = (f^{\rm qc})^{\#}$ via the definition of the upper recession function in \eqref{eq:upperrecession}.  We take for $R>0$ a truncation function $T_R \in C_c^{\infty}(\Rmn)$ with $0 \leq T_R(A) \leq 1$ and $T_R \equiv 1$ on $B_R(0)$ and bound
\[
f^{\rm qc}(x,A) \leq T_R(A)f^{\rm qc}(x,A)+T_R^c(A)g(x,A),
\]
with $T_R^c(A):=1-T_R(A)$. The first integrand on the right-hand side has zero recession function, whereas the second integrand is jointly upper semicontinuous with upper recession function $g^{\#}=(f^{\rm qc})^{\#}$. Applying Theorem~\ref{th:fundyoung} and \eqref{eq:lowerfundyoung} then gives, in combination with the fact that $(\Da u_j)_j \subset \Mcal(\overline{\Omega};\Rmn)$ generates the elementary Young measure $\delta[\Da u]$ (cf.~Remark~\ref{rem:areastrict}),
\begin{align*}
\liminf_{j \to \infty} &\int_{\Omega} f^{\rm qc}(x,\na u_j)\,dx \\
& \leq \lim_{j \to \infty} \int_{\Omega} T_R(\na u_j)f^{\rm qc}(x,\na u_j)\,dx + \limsup_{j \to \infty} \int_{\Omega} T_R^c(\na u_j)g(x,\na u_j)\,dx\\
&\leq \int_{\Omega}T_R(\na u)f^{\rm qc}(x,\na u)+T_R^c(\na u)g(x,\na u)\,dx+\int_{\overline{\Omega}}(f^{\rm qc})^{\#}\left(x,\frac{d \Da_s u}{d\abs{\Da_s u}}\right)\,d\abs{\Da_s u}.
\end{align*}
Letting $R \to \infty$, using the dominated convergence theorem and adding the limit in \eqref{eq:compconv} results in
\[
\liminf_{j \to \infty} \overline{\Fcal}_{\alpha}(u_j) \leq \overline{\Fcal}_{\alpha}(u).
\]
Therefore, if we find for each $j \in \N$ a recovery sequence for $u_j$, then we can conclude the result using a diagonal argument; here, the coercivity of $f$ is important to be able to extract convergent diagonal sequences. We can restrict to the case $u \in g+ C_c^{\infty}(\Omega;\R^n)$ from now on. 

The remaining argument is an adaptation of \cite[Theorem~1.2]{KrS22} to the linear growth setting. To prove the upper bound when $u \in g+C_c^{\infty}(\Omega;\R^m)$, we take a Lipschitz domain $O \Subset \Omega$ and apply Proposition~\ref{prop:connection}\,$(i)$ to find a $v \in W^{1,1}(O;\R^m)$ such that 
\begin{equation}\label{eq:gradequal}
\nabla v = \na u \ \ \text{on $O$.}
\end{equation} 
Then, we apply a classical relaxation theorem \cite[Theorem~9.8]{Dacorogna} to find a sequence $(v_k)_k \subset W^{1,1}(O;\R^m)$ with the same trace values as $v$ on the boundary $\partial \Omega$ such that $v_k \to v$ in $L^1(O;\R^m)$ and
\begin{equation}\label{eq:recoveryclas}
\lim_{k \to \infty}\int_{O} f(x,\nabla v_k)\dx = \int_{O} f^{\rm qc}(x,\nabla v)\dx.
\end{equation}
In view of the coercivity of $f$ we may also suppose
 that $v_k \starto v$ in $BV(O;\R^m)$. Now
  define the auxiliary sequence $(\tilde{v}_k)_k \subset W^{1,1}(\R^n;\R^m)$ via 
$\tilde{v}_k:=v_k-v$ on $O$ and $\tilde{v}_k =0$ in $O^c$. 
By Proposition~\ref{prop:connection}\,$(ii)$ and \cite[Eq.~(3.3)]{KrS22}, we find that the sequence $(\tilde{u}_k)_k \subset \Sone(\R^n;\R^m)$ 
defined by $\tilde{u}_k =  \dac\tilde{v}_k$ satisfies
\begin{equation}\label{eq:auxconvergence}
\tilde{u}_k \to 0 \ \ \text{in $L^1(\R^n;\R^m)$ as $k \to \infty$}
\end{equation}
and its fractional gradients are given by
\begin{equation}\label{eq:gradientequality}
\na \tilde{u}_k = \nabla (v_k -v) \ \ \text{in $O$} \quad \text{and} \quad \na \tilde{u}_k =0 \ \ \text{in $O^c$}.
\end{equation}
Take a cut-off function $\chi \in C_c^{\infty}(\Omega)$ such that $0 \leq \chi \leq 1$ and $\chi|_{O} \equiv 1$. Then, we define the sequence $(w_k)_k \subset \Sonegom$ by
\[
w_k = u+\chi \tilde{u}_k \starto u \ \ \text{in $\bvagom$ as $k \to \infty$},
\]
where the convergence follows from \eqref{eq:auxconvergence} and the Leibniz rule (Lemma~\ref{le:leibniz}). Moreover, we have by \eqref{eq:leibnizbound} the convergence of the residuals
\begin{equation}\label{eq:residual}
R_k:= \na w_k - \na u - \chi \na \tilde{u}_k \to 0 \ \ \text{in $L^1(\R^n;\Rmn)$}.
\end{equation}
In $O$, we find in view of \eqref{eq:gradequal} and \eqref{eq:gradientequality} that $\na w_k = \nabla v_k+R_k$. Due to the strong convergence of $R_k$ to zero it follows by testing with the Lipschitz basis from \cite[Lemma~3]{KrR10a} that the sequences $(\nabla v_k)_k$ and $(\na w_k)_k$ when restricted to $O$ generate (up to a non-relabeled subsequence) the same generalized Young measure $\nu \in Y(O;\Rmn)$. As a result, we use Theorem~\ref{th:fundyoung} twice to conclude
\begin{align}\label{eq:estimate1}
\begin{split}
\liminf_{k \to \infty} \int_{O} f(x,\na w_k)\dx &\leq \int_{O}\inner{f(x,\cdot),\nu_x}\dx + \int_{\overline{O}}\inner{f^{\infty}(x,\cdot),\nu^{\infty}_x}\,d\lambda_{\nu}\\
&= \lim_{k \to \infty}\int_{O}f(x,\nabla v_k)\dx =\int_{O} f^{\rm qc}(x,\nabla v)\dx=\int_{O} f^{\rm qc}(x,\na u)\dx,
\end{split}
\end{align}
where we use \eqref{eq:recoveryclas} and \eqref{eq:gradequal} in the final two equalities. Additionally, in $O^c$ we have that $\na w_k = \na u +R_k$ thanks to \eqref{eq:gradientequality}. Hence, we find using \eqref{eq:residual} and \eqref{eq:growth} that
\begin{equation}\label{eq:estimate2}
\limsup_{k \to \infty}\int_{\Omega \setminus O}f(x,\na w_k)\dx = \limsup_{k \to \infty} \int_{\Omega \setminus O} f(x,\na u +R_k)\dx \leq \norm{M\abs{\na u}+a}_{L^{1}(\Omega \setminus O)}.
\end{equation}
Finally, using Lebesgue's dominated convergence theorem and \eqref{eq:residual} we derive
\begin{equation}\label{eq:estimate3}
\lim_{k \to \infty}\int_{\Omega^c}f(x,\na w_k)\dx = \lim_{k \to \infty} \int_{\Omega^c} f(x,\na u +R_k)\dx=\int_{\Omega^c} f(x,\na u)\dx.
\end{equation}
Summing \eqref{eq:estimate1}, \eqref{eq:estimate2} and \eqref{eq:estimate3} together, we obtain
\[
\liminf_{k \to \infty} \int_{\R^n}f(x,\na w_k)\dx \leq \int_{O}f^{\rm qc}(x,\na u)\dx+\int_{\Omega^c}f(x,\na u)\dx + \norm{M\abs{\na u}+a}_{L^{1}(\Omega \setminus O)},
\] 
which yields the result if we let $O \uparrow \Omega$ and extract a diagonal sequence.
\end{proof}
\begin{remark}\label{rem:relaxation}
a) Because of the coercivity condition of $f$, the functional $\Fcal^{\rm rel}_{\alpha}$ is in particular weak* lower semicontinuous on $\bvagom$. Interestingly, this fact does not immediately follow from the lower semicontinuity result in Theorem~\ref{th:sufficiency}, since the strong recession function of $\mathbbm{1}_{\Omega}f^{\rm qc}+\mathbbm{1}_{\Omega^c}f$ need not exist. An application of the direct method as in Corollary~\ref{cor:existence} provides the existence of minimizers of $\Fcal_{\alpha}^{\rm rel}$. \smallskip

b) A simple argument using the theory of Young measures shows that the functional
\[
\overline{\Fcal}_{\alpha}(u)=\int_{\R^n} f(x,\na u)\dx+\int_{\overline{\Omega}}f^{\infty}\left(x,\frac{d\Da_s u}{\abs{d\Da_s u}}\right)\,d\abs{\Da_su} \quad \text{for $u \in \bvagom$},
\]
is the area-strictly continuous extension of $\Fcal_{\alpha}$ to $\bvagom$. This immediately implies that this functional is also the relaxation of $\Fcal_{\alpha}$ if $f(x,\cdot)$ is quasiconvex for a.e.~$x \in \Omega$, given the lower semicontinuity result from Theorem~\ref{th:sufficiency} and the density with respect to area-strict convergence. \smallskip

c) The requirement \eqref{eq:technicalcond} on $f^{\rm qc}$ is needed for the application of the Jensen's inequalities in the lower bound and allows the relaxation result to be phrased for general Carath\'{e}odory integrands. However, one can dispose of this assumption if we assume a continuity condition similarly as in \cite[Theorem~1.7]{ADR}, that is,
\[
\abs{f(x,A)-f(y,A)} \leq \omega(\abs{x-y})(1+\abs{A}) \quad \text{for all $x,y \in \overline{\Omega}$ and $A \in \Rmn$},
\]
where $\omega:[0,\infty) \to [0,\infty)$ is a continuous and increasing function with $\omega(0)=0$. Indeed, one can utilize \eqref{eq:growth} and \eqref{eq:coercivity} as in \cite[Theorem~7.6]{Rindler} to deduce that $f^{\rm qc}$ inherits the same continuity condition (up to a different modulus of continuity), after which \eqref{eq:technicalcond} readily follows.
\end{remark}
\section*{Acknowledgements}
The author would like to thank Carolin Kreisbeck for the useful suggestions on a preliminary version of the manuscript.

\addcontentsline{toc}{section}{\protect\numberline{}References}
\bibliographystyle{abbrv}
\bibliography{LinearGrowthBib}
\end{document}